\documentclass[11pt,reqno]{article}

\usepackage[utf8]{inputenc}
\usepackage[english]{babel}

\usepackage{csquotes}

\usepackage{geometry}  
\usepackage{graphicx}
\usepackage{amssymb}
\usepackage{epstopdf}
\usepackage{amsmath}
\usepackage{amssymb}
\usepackage{epsfig}
\usepackage{graphicx}
\usepackage{changebar}
\usepackage{centernot}
\usepackage{rotating}
\usepackage{amsthm}
\usepackage{caption}
\usepackage{rotating}
\usepackage{tikz}
\usepackage{setspace}
\usepackage{subfigure}
\usepackage{changepage}
\usepackage{mathtools}
\usepackage{mathrsfs}
\usepackage{youngtab}
\usepackage{color}
\usepackage{mdframed}
\usepackage{bm}
\usepackage{fancyhdr}
\usepackage{lastpage}
\usepackage{txfonts}
\usepackage{slashed}
\usepackage{comment}
\usepackage{longtable}
\usepackage[hang,flushmargin]{footmisc}

\usepackage{blkarray}
\usetikzlibrary{arrows,patterns,decorations}

\usepackage{chngcntr}
\counterwithin{figure}{section}
\counterwithin{table}{section}
\counterwithin{equation}{section}

\usepackage[utf8]{inputenc}
\usepackage{lipsum}
\usepackage[inline, shortlabels]{enumitem}
\setlist{itemjoin ={,\enspace},itemjoin* = {\enspace}}

\usepackage{hyperref}
\hypersetup{
    linkcolor=blue,
    citecolor=blue,
}

\allowdisplaybreaks

\usetikzlibrary{calc}

\makeatletter
\newcommand*\bigcdot{\mathpalette\bigcdot@{.5}}
\newcommand*\bigcdot@[2]{\mathbin{\vcenter{\hbox{\scalebox{#2}{$\m@th#1\bullet$}}}}}
\makeatother

\newenvironment{myfont}{\fontfamily{phv}\selectfont}{\par}
\def\beg{\begin{myfont}}
\def\en{\end{myfont}}

\usetikzlibrary{positioning}

\def\tb{\textbf}

\def\bp{\begin{proof}}
\def\ep{\end{proof}}

\def\be{\begin{enumerate}}
\def\ee{\end{enumerate}}

\def\bi{\begin{itemize}}
\def\ei{\end{itemize}}

\def\multiset#1#2{\ensuremath{\left(\kern-.3em\left(\genfrac{}{}{0pt}{}{#1}{#2}\right)\kern-.3em\right)}}

\newtheorem{lemma}{Lemma}[section]
\newtheorem{theorem}[lemma]{Theorem}
\newtheorem{corollary}[lemma]{Corollary}

\newtheorem{proposition}[lemma]{Proposition}

\theoremstyle{definition}
\newtheorem{definition}[lemma]{Definition}
\newtheorem{remark}[lemma]{Remark}

\title{Classifications of $\Gamma$-colored $d$-complete posets and \\ upper $P$-minuscule Borel representations\thanks{This main results of this paper formed part of a 2019 doctoral thesis \cite{Str} written under the supervision of Robert A. Proctor at the University of North Carolina.}}

\author{Michael C. Strayer \\ Hampden--Sydney College \\ Hampden--Sydney, VA 23943 U.S.A. \\ mstrayer@hsc.edu}
\date{August 15, 2020}                                           

\begin{document}

\maketitle 

\begin{spacing}{1.1}

    
    
    
    
    
    
    
    
    
    



\begin{abstract}
The $\Gamma$-colored $d$-complete posets correspond to certain Borel representations that are analogous to minuscule representations of semisimple Lie algebras.
We classify $\Gamma$-colored $d$-complete posets which specifies the structure of the associated representations.  
We show that finite $\Gamma$-colored $d$-complete posets are precisely the dominant minuscule heaps of J.R. Stembridge.  These heaps are reformulations and extensions of the colored $d$-complete posets of R.A. Proctor.  
We also show that connected infinite $\Gamma$-colored $d$-complete posets are precisely order filters of the connected full heaps of R.M. Green.
\end{abstract}

\vspace{.2in} 



\vspace{-.2in}



















\end{spacing}


\begin{spacing}{1.25}

\section{Introduction}

The main objects of study in this paper are locally finite partially ordered sets that are colored by the nodes of a Dynkin diagram.  Each Dynkin diagram corresponds to a Kac--Moody algebra, its Borel subalgebra, and a Weyl group.  Our motivations include representation theory, but our techniques are entirely combinatorial.

This paper and \cite{PaperClassifyMinuscule} are sequels to \cite{Unify}.  In that paper and \cite{Str}, we introduced the $\Gamma$-colored $d$-complete and $\Gamma$-colored minuscule posets.  These definitions were made without regard to finite or infinite poset cardinality.  The main result of \cite{Unify} showed that these posets correspond precisely to representations of Kac--Moody algebras (or subalgebras) that generalize the minuscule representations of semisimple Lie algebras.  In this paper we classify the $\Gamma$-colored $d$-complete posets, and in \cite{PaperClassifyMinuscule} we classify the $\Gamma$-colored minuscule posets.  These classifications thus provide the possible structures of the above representations.  For connected posets, these classifications are summarized in Table \ref{RepClass}.  We describe the posets appearing in the bottom row of this table in the next five paragraphs and their representations in Section \ref{RepInteractions}.

\begin{table}[t!]
    \centering
    \begin{tabular}{|l||c|c|}
        \hline 
         \textbf{Connected posets} & \textbf{Finite} & \textbf{Infinite} \\
         \hline 
         \hline 
        \textbf{$\Gamma$-colored minuscule} & Colored minuscule posets & Full heaps \\
        First introduced by: & R.A. Proctor (1984) & R.M. Green (2007) \\
        \hline 
        \textbf{$\Gamma$-colored $d$-complete} & Dominant minuscule heaps & Filters of full heaps \\
        First introduced by: & J.R. Stembridge (2001) & This author (2019) \\
        \hline
    \end{tabular}
    \caption{The classifications of connected $\Gamma$-colored minuscule and $\Gamma$-colored $d$-complete posets}
    \label{RepClass}
\end{table}

Working in the simply laced setting, R.A. Proctor defined \cite{Wave} and classified \cite{DDCT} the $d$-complete posets.
He defined both colored and uncolored versions of these finite posets using properties that govern the structure of certain intervals within the poset. 
J.R. Stembridge used the notion of a heap to refine Proctor's definition with a new set of coloring axioms while also extending to the multiply laced setting in \cite{Ste}.  He called these new colored posets dominant minuscule heaps.
Proctor and Stembridge developed these posets to study certain ``$\lambda$-minuscule'' Weyl group elements introduced by D. Peterson \cite{Car}.
Section 10 of \cite{Wave} and Section 3 of \cite{Ste} respectively describe the correspondence between their colored posets and these Weyl group elements.  Linear extensions of these posets correspond to reduced expressions for their corresponding $\lambda$-minuscule Weyl group elements. 
Colored $d$-complete posets and dominant minuscule heaps generalize the colored minuscule posets of \cite{BLPP} appearing in the top left entry of Table \ref{RepClass}.

Proctor classified the $d$-complete posets in \cite{DDCT} working in the uncolored setting.  He demonstrated how each connected $d$-complete poset could be written as a ``slant sum'' of ``slant irreducible'' $d$-complete posets, and then he classified these slant irreducible $d$-complete posets into fifteen families.  Given any uncolored $d$-complete poset, a unique colored $d$-complete poset can be produced (and vice versa), so this classified the colored $d$-complete posets as well.  Thus Stembridge only needed to classify the slant irreducible dominant minuscule heaps in the multiply laced setting.  He did so in \cite{Ste}, producing two new families of slant irreducible dominant minuscule heaps.  We outline this slant sum decomposition in Section \ref{SectionClassifydComplete}.

Proctor's $d$-complete posets have been used extensively; see Section 12 of \cite{ProScop}.  That paper extended uncolored $d$-complete posets to the infinite (but locally finite) setting.  Proctor has also considered the related question of defining infinite colored $d$-complete posets (personal communication).  That question was answered in \cite{Str,Unify} by defining $\Gamma$-colored $d$-complete posets since they generalize the finite colored $d$-complete posets and were created for their application to Kac--Moody representation theory.  K. Nakada has also proposed an answer to this question using coroots of a Kac--Moody algebra; see \cite{Nak}.  The connected infinite versions of $\Gamma$-colored $d$-complete posets are order filters of connected full heaps of R.M. Green.

Green developed full heaps as infinite analogs of the minuscule heaps of Stembridge using several similar properties.  The primary difference is cardinality; while dominant minuscule heaps are finite, full heaps have the property that every subposet consisting of all elements of a given color is isomorphic to $\mathbb{Z}$.  Therefore full heaps are unbounded above and below.  Similarly to colored $d$-complete posets and dominant minuscule heaps, full heaps were developed with algebraic applications in mind.  Green used them to construct representations of affine Kac--Moody algebras in \cite{Gre1} and representations of affine Weyl groups in \cite{Gre2}.   They formed the main objects of study in his Cambridge monograph \cite{Gre}, wherein he used them in many applications to representation theory and algebraic geometry.  Green noted that full heaps are infinite analogs of colored minuscule posets and representations of affine algebras constructed with full heaps are infinite-dimensional analogs of finite-dimensional minuscule representations of semisimple Lie algebras.  

Green classified all full heaps colored by affine Dynkin diagrams in Theorem 6.6.2 of \cite{Gre}.  Z.S. McGregor-Dorsey showed in his doctoral thesis \cite{McD} under Green that the connected components of finite Dynkin diagrams that color full heaps must have affine type.  Therefore the list produced by Green is a complete list of all full heaps colored by connected Dynkin diagrams with finitely many nodes.


Our main result is the classification of all $\Gamma$-colored $d$-complete posets in Theorem \ref{TheoremClassify}.  We first show finite $\Gamma$-colored $d$-complete posets are precisely the dominant minuscule heaps of Stembridge in Theorem \ref{TheoremMSequivJRS}.  We show connected infinite $\Gamma$-colored $d$-complete posets are precisely order filters of the connected full heaps of Green in Theorem \ref{TheoremClassifyInfinite}.  We use these results and apply the previous classifications of Proctor, Stembridge, Green, and McGregor-Dorsey to obtain the classification of all $\Gamma$-colored $d$-complete posets.

Green's work with Kac--Moody representations built from colored posets in \cite{Gre1,Gre} was the principal antecedent and inspiration for our paper \cite{Unify}. 
Working in the simply laced case, that paper built upon the representation results of Green; one of its main results showed that $\Gamma$-colored $d$-complete posets are both necessary and sufficient to build ``upper $P$-minuscule'' Borel representations from colored posets.  The main results in \cite{Unify} also appeared in \cite{Str} which included the multiply laced case.
For finite posets, upper $P$-minuscule representations are closely related to Demazure modules indexed by $\lambda$-minuscule Weyl group elements for dominant integral weights $\lambda$.  For infinite posets, these representations are new. 

We give colored poset definitions in Section \ref{SectionDefinitions}.  Section \ref{PropertyRelationships} is dedicated to developing general coloring property relationships needed in the rest of the paper; there we obtain the equivalence between finite $\Gamma$-colored $d$-complete posets and dominant minuscule heaps.  Propositions \ref{PropProctorProof} and \ref{PropLCB2} in Section \ref{SectionTwoImportantPropositions} are used heavily in Section \ref{MainProof}, wherein we show connected infinite $\Gamma$-colored $d$-complete posets are precisely order filters of connected full heaps.  The classification of $\Gamma$-colored $d$-complete posets is given in Section \ref{SectionClassifydComplete} after describing in more depth the previous work of Proctor, Stembridge, Green, and McGregor-Dorsey.  We connect our work to $\lambda$-minuscule Weyl group elements and Kac--Moody representation theory in Section \ref{RepInteractions}.



\section{Colored poset definitions}\label{SectionDefinitions}

Let $P$ be a partially ordered set.  We follow \cite{Sta} for the following commonly used terms: interval, covering relations and the Hasse diagram, order ideal and order filter, saturated chain, antichain, linear extension, and order dual poset $P^*$. 
We use letters such as $z,y,x,\dots$ to denote elements of $P$.  Let $x,y \in P$.  If $x$ is covered by $y$, then we write $x \to y$.  We say that $x$ and $y$ are \emph{neighbors} if $x \to y$ or $y \to x$.  If $x \le y$, then $(x,y)$ and $[x,y]$ are respectively the open and closed intervals between $x$ and $y$.  We require $P$ to be \emph{locally finite}, meaning that all intervals in $P$ are finite.

Let $r \ge 1$ and let $P_1,\dots,P_r$ be disjoint posets.  The \emph{disjoint union} of $P_1,\dots,P_r$ is the poset $P = \bigcup_{i=1}^r P_i$ for which $x \le y$ in $P$ when there is some $1 \le k \le r$ such that $x \le y$ in $P_k$.  A poset is \emph{connected} if it cannot be written as the disjoint union of two or more of its nonempty subposets.  If $P$ is the disjoint union of $P_1,\dots,P_r$ and each of the posets $P_1,\dots,P_r$ is connected, then they are the \emph{connected components} of $P$.

Let $\Gamma$ be a finite set.  We use letters such as $a,b,c,\dots$ to denote the elements of $\Gamma$ and call them \emph{colors}.  
We fix integers $\theta_{ab}$ for $a,b \in \Gamma$ subject to the following requirements:
\begin{enumerate}[(i),nosep]
    \item For all $a \in \Gamma$ we have $\theta_{aa} = 2$.
    \item For all distinct $a,b \in \Gamma$, we have $\theta_{ab} \le 0$ and $\theta_{ba} \le 0$.
    \item For all distinct $a,b \in \Gamma$, we have $\theta_{ab} = 0$ if and only if $\theta_{ba} = 0$.
\end{enumerate}
If the elements of $\Gamma$ are ordered, then these requirements make $[\theta_{ab}]$ a \textit{generalized Cartan matrix}; see Section \ref{RepInteractions}.  Let $a,b \in \Gamma$ be distinct.  We say $a$ and $b$ are \emph{distant} when $\theta_{ab} = 0$ and \emph{adjacent} when $\theta_{ab} < 0$.  In the latter case, we write $a \sim b$ and say $a$ is \textit{$k$-adjacent to $b$} (respectively $b$ is \emph{$l$-adjacent to $a$}) when $\theta_{ab} = -k$ (respectively $\theta_{ba} = -l$).

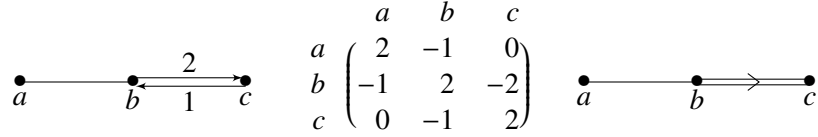
\begin{figure}[t!]
    \centering
    \begin{tikzpicture}[scale=0.75]
    \tikzset{edge/.style = {->,> = latex'}}
    
        \node(A) at (-6,0){$\bullet$};
        \node(B) at (-4,0){$\bullet$};
        \node at (-8,0){$\bullet$};
        \node at (4,0){$\bullet$};
        \node at (6,0){$\bullet$};
        \node at (2,0){$\bullet$};
        
        \node at (-8,-.3){$a$};
        \node at (-6,-.3){$b$};
        \node at (-4,-.3){$c$};
        \node at (2,-.3){$a$};
        \node at (4,-.3){$b$};
        \node at (6,-.3){$c$};
        
        \node at (-5,.35){$2$};
        \node at (-5,-.35){$1$};
        
        \draw (4.9,.15) -- (5.1,0) -- (4.9,-.15);
        
        \draw (4,.05) -- (6,.05);
        \draw (4,-.05) -- (6,-.05);
        \draw (-8,0) -- (-6,0);
        \draw (2,0) -- (4,0);
        
        \draw[edge] (-6,0.075)  to (-4.05,.075);
       \draw[edge] (-4,-0.075) to (-5.95,-.075);
        
        \node at (-1,0){$\begin{blockarray}{rrrr} 
         & a & b & c \\
         \begin{block}{r(rrr)}
        a & 2 & -1 & 0 \\
        b & -1 & 2 & -2 \\
        c & 0 & -1 & 2 \\
        \end{block}
        \end{blockarray}$};
    \end{tikzpicture}
    \caption{From left-to-right: A Dynkin diagram $\Gamma$, a generalized Cartan matrix for $\Gamma$ (using the transpose convention of \cite{Ste}), and the Dynkin diagram of finite type $B_3$ traditionally used for $\Gamma$.}
    \label{DynkinFigure}
\end{figure}

We describe an equivalent way to view $\Gamma$ and the integers $\theta_{ab}$ for $a,b \in \Gamma$ as a graph with finitely many nodes and no loops.
Let $a,b \in \Gamma$ be distinct.  If $a$ and $b$ are distant, then there is no edge between $a$ and $b$.  Now suppose $a \sim b$.  If $\theta_{ab}\theta_{ba} = 1$, then there is a single undirected edge between $a$ and $b$.  
If $\theta_{ab}\theta_{ba} > 1$, then there is a directed edge from $a$ to $b$ (respectively from $b$ to $a$) decorated with the integer $-\theta_{ab}$ (respectively $-\theta_{ba}$).
The graph $\Gamma$ is a \textit{Dynkin diagram}.  See Figure \ref{DynkinFigure}; in that figure, the color $b$ is $1$-adjacent to the color $a$ and $2$-adjacent to the color $c$.  The Dynkin diagram $\Gamma$ is \emph{acyclic} if the underlying simple graph, obtained by replacing each pair of directed edges by a single undirected edge, is acyclic.  If $\Gamma$ is a simple graph, then $\theta_{ab} \in \{-1,0,2\}$ for all $a,b \in \Gamma$ and we say $\Gamma$ is \emph{simply laced}.  Otherwise $\Gamma$ is \emph{multiply laced}.

We equip $P$ with a surjective \emph{coloring function} $\kappa : P \to \Gamma$ onto the nodes of $\Gamma$ and say that $P$ is \emph{$\Gamma$-colored}.  We define properties that $P$ may satisfy with respect to this coloring:
\begin{itemize}[nosep]
    \item [] (EC) Elements with equal colors are comparable.
    \item [] (NA) Neighbors have adjacent colors.
    \item [] (AC) Elements with adjacent colors are comparable.
\end{itemize}
For each color $a$, we define $P_a := \{ x \in P \ | \ \kappa(x) = a\}$.  Let $a \in \Gamma$ and $x,y \in P$ with $x < y$.  We say $x$ and $y$ are \emph{consecutive elements of the color $a$} if $x,y \in P_a$ and $(x,y)$ contains no elements of the color $a$.
\begin{itemize}[nosep]
    \item [] (ICE2) For every $a \in \Gamma$, if $x < y$ are consecutive elements of the color $a$, then $\sum_{z \in (x,y)} -\theta_{\kappa(z),a} = 2$.
\end{itemize}
For a subset $S \subseteq P$ and element $x \in S$, we define the set $U(x,S) := \{y \in S \ | \ \text{$y > x$ and $\kappa(y) \sim \kappa(x)$}\}$.  Dually, define the set $L(x,S) := \{y \in S \ | \ \text{$y < x$ and $\kappa(y) \sim \kappa(x)$}\}$.  These sets will be used in Section \ref{SectionTwoImportantPropositions} when $S$ is an ideal or filter of $P$, but for now we only need the case $S = P$.  Let $k \ge 1$.
\begin{itemize}[nosep]
    \item [] (UCB$k$) For every $a \in \Gamma$, if $x$ is maximal in $P_a$, then $U(x,P)$ is finite and $\sum_{y \in U(x,P)} -\theta_{\kappa(y),a} \le k$.
    \item [] (LCB$k$) For every $a \in \Gamma$, if $x$ is minimal in $P_a$, then $L(x,P)$ is finite and $\sum_{y \in L(x,P)} -\theta_{\kappa(y),a} \le k$.
\end{itemize}

\begin{remark}
The properties ICE2, UCB$k$, and LCB$k$ for $k \ge 1$ all restrict the ``census'' of elements with colors that are adjacent to some given color within certain convex subsets of $P$.  For a given color $a$ and convex subset of $P$, a member $y$ of that convex subset is included in the census if $\kappa(y) \sim a$.  The count of the ``household'' of $y$ is $-\theta_{\kappa(y),a}$.  For example, the property ICE2 says  the census within an interval between consecutive elements of a given color must be two.
These properties were defined in \cite{Unify} for simply laced Dynkin diagrams and were respectively called I2A, Mx$k$GA, and Mn$k$LA.  We have changed their names here to reflect the interval census being equal to two and upper census and lower census being bounded by $k$.  The properties UCB$k$ and LCB$k$ are the \emph{frontier census properties}.  The property LCB2 plays a key role in Section \ref{MainProof}; we give sufficient conditions for this property in Proposition \ref{PropLCB2}.  The most important frontier census property is UCB1 since it is one of the properties required for a poset to be $\Gamma$-colored $d$-complete.
\end{remark}

We use the above properties to give the main colored poset definitions in this paper.
\begin{definition}\label{Gcdc}
A \emph{$\Gamma$-colored $d$-complete} poset is a locally finite $\Gamma$-colored poset (of any cardinality) that satisfies EC, NA, AC, ICE2, and UCB1.  If it also satisfies LCB1, then it is a \emph{$\Gamma$-colored minuscule} poset.
\end{definition}

Stembridge defined the dominant minuscule heaps in \cite{Ste}.  We provide this definition, translated to our conventions for notation and terminology.

\begin{definition}\label{dominantminusculeheap}
A \emph{dominant minuscule heap} is a finite $\Gamma$-colored poset that satisfies
\begin{itemize}[nosep]
    \item [] (S1) All neighbors in $P$ have colors that are equal or adjacent in $\Gamma$, and the colors of incomparable elements are distant.
    \item [] (S2) For every $a \in \Gamma$, the open interval between any two consecutive elements of color $a$ either contains (i) exactly two elements whose colors are adjacent to $a$, and their colors are 1-adjacent to $a$, or (ii) exactly one element, and the color of this element is 2-adjacent to $a$.
    \item [] (S3) For every $a \in \Gamma$, an element that is maximal in $P_a$ is covered by at most one element, and this element is maximal among all elements of some color that is 1-adjacent to $a$.
    \item [] (S4) The Dynkin diagram $\Gamma$ is acyclic.
\end{itemize}
\end{definition}


\noindent Stembridge does not require $\kappa$ to be surjective, and his final property is merely that the colors appearing in $P$ index an acyclic subset of $\Gamma$.  However, when he classifies the dominant minuscule heaps in \cite{Ste}, he does so assuming that $\kappa$ is surjective.  
Since we only use his combinatorial classification, there is no harm in continuing to assume $\kappa$ is surjective and using S4 in the form presented here.  All of Stembridge's properties may be satisfied by a poset of any cardinality.

Inspired by Stembridge's dominant minuscule heaps, Green introduced a new class of infinite $\Gamma$-colored posets in \cite{Gre1,Gre2}; these posets were the central object in \cite{Gre}.  We provide the definition given in \cite{Gre}, also translated to our conventions for notation and terminology.

\begin{definition}\label{fullheap}
A \emph{full heap} is a locally finite $\Gamma$-colored poset that satisfies
\begin{itemize}[nosep]
    \item [] (G1) For every $a \in \Gamma$ and every pair of adjacent colors $b,c \in \Gamma$, the subsets $P_a$ and $P_b \cup P_c$ of $P$ are chains in $P$.
    \item [] (G2) The partial order on $P$ is the minimal partial order extending the given partial order on the above chains $P_a$ and $P_b \cup P_c$.
    \item [] (G3) For every $a \in \Gamma$, the set $P_a$ is isomorphic as a poset to $\mathbb{Z}$.
    \item [] (G4) For every $a,b \in \Gamma$, if $a \sim b$ and $x \in P_a$, then $x$ has a neighbor in $P_b$.
    \item [] (G5) For every $a \in \Gamma$, if $x < y$ are consecutive elements of the color $a$, then $\sum_{z \in [x,y]} \theta_{\kappa(z),a} = 2$.
\end{itemize}
\end{definition}

\noindent Green does not require $\Gamma$ to have finitely many nodes.  He also does not require $\kappa$ to be surjective when he first defines a \emph{heap} (any poset satisfying G1 and G2); note the surjectivity of $\kappa$ follows from G3.  The properties G1, G2, G4, and G5 may be satisfied by a poset of any cardinality.


\section{Coloring property relationships and finite \texorpdfstring{$\Gamma$}{Gamma}-colored \texorpdfstring{$d$}{d}-complete posets}\label{PropertyRelationships} 

In this section we establish the relationships we need between various poset coloring properties.  
Corollary \ref{CorFullHeapEquiv} establishes a new characterization of full heaps, and Theorem \ref{TheoremMSequivJRS} shows that finite $\Gamma$-colored $d$-complete posets are precisely the dominant minuscule heaps of Stembridge.  Our first result shows that the definition of full heap can be shortened.

\begin{lemma}\label{LemRedundant}
Let $P$ be a $\Gamma$-colored poset that satisfies G1, G2, G3, and G5.  Then $P$ also satisfies G4.  Hence G4 can be removed from the definition of full heap.
\end{lemma}

\begin{proof}
Assume for a contradiction that $P$ fails $G4$.  Then there is a pair of adjacent colors $a,b \in \Gamma$ and an element $x \in P_a$ that has no neighbor in $P$ of color $b$.  By G3 the set $P_b$ is isomorphic to $\mathbb{Z}$, and by G1 all elements of $P_b$ are comparable to $x$.  
Then by local finiteness, the subset of $P_b$ consisting of elements less than $x$ of color $b$ is nonempty and bounded above; let $u$ be its maximal element.  Similarly, the subset of $P_b$ consisting of elements greater than $x$ of color $b$ is nonempty and bounded below; let $v$ be its minimum element.
Then $u < v$ are consecutive elements of the color $b$.  
Again using local finiteness, let $u \to y \to \cdots \to x \to \cdots \to z \to v$ be a saturated chain from $u$ to $v$ in $P$.
Since $x$ is not a neighbor to $u$ or $v$, we know $x \ne y$ and $x \ne z$.  By G2 all covering pairs in $P$ have equal or adjacent colors in $\Gamma$.  By G5 the colors of covering pairs cannot be equal.  Thus $b \sim \kappa(y)$ and $b \sim \kappa(z)$.  By assumption we also know $b \sim a = \kappa(x)$.  Hence $\theta_{\kappa(y),b} \le -1$ and $\theta_{\kappa(z),b} \le -1$ and $\theta_{\kappa(x),b} \le -1$.  The only positive numbers in the sum $\sum_{w \in [u,v]} \theta_{\kappa(w),b}$ are $\theta_{\kappa(u),b} = \theta_{\kappa(v),b} = 2$.  Thus $\sum_{w \in [u,v]} \theta_{\kappa(w),b} < 2$, contradicting G5.
\end{proof}

Finite posets satisfying the properties in the next result are the ``minuscule heaps'' of Stembridge \cite{Ste}.

\begin{proposition}\label{PropMinHeapEquiv}
Let $P$ be a $\Gamma$-colored poset.  The following are equivalent: 
\begin{enumerate}[(i),nosep]
    \item The poset $P$ satisfies EC, NA, AC, and ICE2.
    \item The poset $P$ satisfies G1, G2, and G5.
    \item The poset $P$ satisfies S1 and S2.
\end{enumerate}
\end{proposition}

\begin{proof}
First suppose that (i) holds.  Then EC and AC imply G1.  By local finiteness, the order on $P$ is the reflexive transitive closure of its covering relations, that is, the minimal order given its pairs of ordered neighbors.  By NA, each pair of neighbors belongs to some chain $P_b \cup P_c$ for adjacent $b,c \in \Gamma$.  Thus G2 holds.  Let $a \in \Gamma$ and let $x < y$ be consecutive elements of the color $a$.  Since $\theta_{\kappa(x),a} = \theta_{\kappa(y),a} = 2$, the equalities in ICE2 and G5 are equivalent.  Thus (ii) holds.

Now suppose (ii) holds.  Then G1 implies incomparable elements have distant colors and G2 implies neighbors have equal or adjacent colors.  Thus S1 holds.  Let $a \in \Gamma$ and let $x < y$ be consecutive elements of the color $a$.  Since the equalities in ICE2 and G5 are equivalent, we have $\sum_{z \in (x,y)} -\theta_{\kappa(z),a} = 2$.  The integers $-\theta_{\kappa(z),a}$ for $z \in (x,y)$ are non-negative, so the sum $\sum_{z \in (x,y)} -\theta_{\kappa(z),a}$ is either $1 + 1$ or the single term 2.  If the sum is $1+1$, then S2(i) holds.  Now suppose the sum is the single term 2, so that there is only one element $w \in (x,y)$ with color adjacent to $a$.  Since neighbors have equal or adjacent colors by G2, we see $(x,y) = \{w\}$.  Hence S2(ii) is satisfied.  Thus S2 holds, so (iii) holds.

Now suppose (iii) holds.  Using S2 we see that neighbors do not have equal colors, so S1 implies that EC, NA, and AC hold.  Finally, S2 implies ICE2, so (i) holds.
\end{proof}





The previous results allow us to obtain a new characterization of full heaps.

\begin{corollary}\label{CorFullHeapEquiv}
Let $P$ be a $\Gamma$-colored poset.  The following are equivalent:
\begin{enumerate}[(i),nosep]
    \item The poset $P$ satisfies EC, NA, AC, ICE2, and G3.
    \item The poset $P$ is a full heap.
\end{enumerate}
\end{corollary}

\begin{proof}
The property G3 can be added to each set of properties in Proposition \ref{PropMinHeapEquiv}(i) and \ref{PropMinHeapEquiv}(ii) to produce two new equivalent sets of properties.  By Lemma \ref{LemRedundant}, the latter set of properties defines a full heap.
\end{proof}

By restricting our attention to finite posets, we now show UCB1 is a reformulation of the axioms S3 and S4 of Stembridge when one or both of the properties NA and AC are assumed.

\begin{proposition}\label{PropFrontierCensusEquivtoS3S4}
Let $P$ be a finite $\Gamma$-colored poset that satisfies NA.
\begin{enumerate}[(a),nosep]
    \item If $P$ satisfies S3 and S4, then it satisfies UCB1.
    \item If $P$ satisfies UCB1 and also AC, then it satisfies S3 and S4.
\end{enumerate}
\end{proposition}

\begin{proof}
To show (a), assume for a contradiction that $P$ fails UCB1.  Then there is a color $a \in \Gamma$ and an element $x$ that is maximal in $P_a$ such that $\sum_{y \in U(x,P)} -\theta_{\kappa(y),a} > 1$.  Let $F$ be the filter generated by $x$.  By repeatedly applying S3 if necessary, we see that $F$ is a chain with no repeated colors.  Since $\sum_{y \in U(x,P)} -\theta_{\kappa(y),a} > 1$, the set  $U(x,P)$ is nonempty.  Let $z$ be maximal in $U(x,P)$ and define $b := \kappa(z)$.  If $z$ covers $x$, then it is the only element greater than $x$ with color adjacent to $a$.  This implies that $-\theta_{ba} > 1$, which violates S3.  Thus $z$ does not cover $x$.  By NA and the fact that $F$ has no repeated colors, the saturated chain from $x$ to $z$ induces a path in $\Gamma$ from $a$ to $b$ consisting of at least three distinct colors.  Since $a \sim b$, this shows $\Gamma$ is not acyclic, contradicting S4.  Hence $P$ satisfies UCB1.

To show (b), assume $P$ satisfies UCB1 and AC.  Let $a \in \Gamma$ and let $x$ be maximal in $P_a$.  By UCB1 we know $\sum_{y \in U(x,P)} -\theta_{\kappa(y),a} \le 1$.  
This inequality shows that $|U(x,P)| \le 1$.  It also shows that if $z \in U(x,P)$, then $z$ is the maximal element of its color and $\kappa(z)$ is 1-adjacent to $a$.  By NA, any element covering $x$ must be in $U(x,P)$.
Hence S3 holds.  Now assume for a contradiction that $\Gamma$ is not acyclic.  
Fix a cycle $\mathcal{C}$ in $\Gamma$ and define $M := \{ y \in P \ | \ \text{$y$ is maximal in $P_b$ for some $b \in \mathcal{C}$}\}$.  Choose $x$ minimal in $M$ and set $a := \kappa(x)$.  Then $x$ is maximal in $P_a$.  By AC and the minimality of $x$ in $M$, the two elements in $M$ whose colors are adjacent to $a$ in $\mathcal{C}$ are greater than $x$.  
Hence $\sum_{y \in U(x,P)} -\theta_{\kappa(y),a} > 1$, violating UCB1.  Hence $P$ satisfies S4.
\end{proof}

Propositions \ref{PropMinHeapEquiv} and \ref{PropFrontierCensusEquivtoS3S4} combine to obtain the main result of this section, namely, the equivalence between finite $\Gamma$-colored $d$-complete posets and dominant minuscule heaps.  

\begin{theorem}\label{TheoremMSequivJRS}
Let $P$ be a finite $\Gamma$-colored poset.  Then $P$ is a $\Gamma$-colored $d$-complete poset if and only if it is a dominant minuscule heap.
\end{theorem}


\section{Frontier censuses for ideals and filters}\label{SectionTwoImportantPropositions}

In this section we obtain the two most important results used in Section \ref{MainProof}, namely, Propositions \ref{PropProctorProof} and \ref{PropLCB2}.  
They respectively provide structural information for a $\Gamma$-colored $d$-complete poset whose color chains $P_b$ for $b \in \Gamma$ are uniformly bounded above or unbounded above.  
We thus view these propositions as complimentary.
Their proofs also have a similar structure (but are order dual to each other) and use the following upper and lower frontier censuses for ideals and filters of $P$.  Recall that for $S \subseteq P$ and $x \in S$ we have $U(x,S) = \{y \in S \ | \ \text{$y > x$ and $\kappa(y) \sim \kappa(x)$}\}$ and $L(x,S) = \{y \in S \ | \ \text{$y < x$ and $\kappa(y) \sim \kappa(x)$}\}$. 

\begin{definition}\label{DefUpperLowerFrontierSum}
Let $P$ be a $\Gamma$-colored poset that satisfies EC.  Let $I$ be an ideal and $F$ a filter of $P$.  Let $b \in \Gamma$.
\begin{enumerate}[(a),nosep]
    \item If $x$ is maximal in $P_b \cap I$ and $U(x,I)$ is finite, then define $U_b(I) := \sum_{y \in U(x,I)} -\theta_{\kappa(y),b}$.
\item If $x$ is minimal in $P_b \cap F$ and $L(x,F)$ is finite, then define $L_b(F) := \sum_{y \in L(x,F)} -\theta_{\kappa(y),b}$.
\end{enumerate}
We call $U_b(I)$ (respectively $L_b(F)$) the \emph{upper} (respectively \emph{lower}) \emph{frontier census} of $I$ (respectively $F$) for $b$.
\end{definition}


We note that if $P$ satisfies UCB1 and $U_b(P)$ is defined for some $b \in \Gamma$, then $U_b(P) \le 1$.  The following are the key properties of upper and lower frontier censuses used in the proofs of Propositions \ref{PropProctorProof} and \ref{PropLCB2}. 

\begin{lemma}\label{LemUpperFrontier}
Let $P$ be a $\Gamma$-colored $d$-complete poset. Let $I$ be an ideal of $P$ and let $F := P - I$ be its corresponding filter.  Let $a \in \Gamma$.
\begin{enumerate}[(a),nosep]
\item Suppose that $U_a(I)$ is defined.
    \begin{enumerate}[(i),nosep]
    \item We have $U_a(I) \in \{0,1,2\}$.
    \item We have $U_a(I) = 2$ if and only if $F$ contains a minimal element of color $a$.
    \end{enumerate}
\item Suppose that $I$ contains an element of color $a$ and $L_a(F)$ is defined.
\begin{enumerate}[(i),nosep]
\item We have $L_a(F) \in \{0,1,2\}$.
\item We have $L_a(F) = 2$ if and only if $I$ contains a maximal element of color $a$.
\end{enumerate}
\end{enumerate}
\end{lemma}

\begin{proof}
We first prove (a).  By EC the set $P_a$ is a chain.  Since $U_a(I)$ is defined, the set $P_a \cap I$ must have a maximal element; call this element $x$.  Note that $U_a(I)$ is a non-negative integer.  If $P_a \cap F$ is empty,  then $x$ is maximal in $P_a$ and UCB1 implies $U_a(I) \le 1$.  So for both (i) and (ii), we now suppose $P_a \cap F$ is nonempty.  

By local finiteness, the set $P_a \cap F$ must have a minimal element $y$.  Then $x < y$ are consecutive elements of the color $a$.  By AC all elements in $U(x,I)$ or $L(y,F)$ are in $(x,y)$.  Hence $L_a(F)$ is defined, and ICE2 gives $U_a(I) + L_a(F) = 2$.  Since $L_a(F)$ is a non-negative integer, this proves (i).  By NA we see that $L_a(F) = 0$ if and only if there are no elements in $F$ covered by $y$; that is, if and only if $y$ is minimal in $F$.  This proves (ii).  

For (b), since $L_a(F)$ is defined, the set $P_a \cap F$ must have a minimal element.  We are assuming that $P_a \cap I$ is nonempty.  Hence a dualized version of the above paragraph proves (b).
\end{proof}

\begin{remark}\label{RemUpperandLowerSumAdditionalElement}
Suppose we are in the situation of Lemma \ref{LemUpperFrontier}(a) and $U_a(I) = 2$.  Then $F$ contains a minimal element $y$ of color $a$.  Set $I' := I \cup \{y\}$.  Fix $b \in \Gamma$ and suppose $U_b(I)$ is defined.  Then $U_b(I')$ is also defined and $U_b(I') = U_b(I) - \theta_{ab}$.  To see this, note that if $a = b$ then $\theta_{ab} = 2$ and so $U_b(I') = 0 = U_b(I) - \theta_{ab}$. 
If $a \sim b$, then $y \in U(x,I')$ by AC and hence $U(x,I') = U(x,I) \cup \{y\}$.  Thus
\begin{align*} 
U_b(I') = \sum_{z \in U(x,I')} -\theta_{\kappa(z),b} = \sum_{z \in U(x,I)} -\theta_{\kappa(z),b} - \theta_{\kappa(y),b} = U_b(I) - \theta_{ab}.
\end{align*}
If $a$ and $b$ are distant, then $U(x,I') = U(x,I)$ and so $U_b(I') = U_b(I) = U_b(I) - \theta_{ab}$.

Now suppose we are in the situation of Lemma \ref{LemUpperFrontier}(b) and $L_a(F) = 2$.  Then $I$ contains a maximal element $y$ of color $a$.  Set $F' := F \cup \{y\}$.  A dualized argument shows $L_b(F') = L_b(F) - \theta_{ab}$ for every $b \in \Gamma$.
\end{remark}

The first main result of this section shows there are no infinite $\Gamma$-colored $d$-complete posets whose color chains $P_b$ for $b \in \Gamma$ are uniformly bounded above.  

\begin{proposition}\label{PropProctorProof}
Suppose $P$ is a $\Gamma$-colored $d$-complete poset and for every $b \in \Gamma$, the set $P_b$ is bounded above.  Then $P$ is finite.
\end{proposition}

\begin{proof}
Suppose for a contradiction that  $P$ is infinite.  Define $\Gamma'$ to be the set of colors for which infinitely many elements of that color appear in $P$.  Since $\Gamma$ is finite but $P$ is infinite, we see that $\Gamma' \ne \emptyset$.

Let $F$ be the filter of $P$ generated by the minimal elements of $P_a$ for every $a \in \Gamma - \Gamma'$.  Since $P_b$ is bounded above for every $b \in \Gamma$ and $P$ is locally finite, the filter $F$ is finite.  Note that $F$ is empty if $\Gamma' = \Gamma$.  Set $p := |F|$ and note that $I := P - F$ is an infinite ideal of $P$ colored by $\Gamma'$.  Since $P_b \cap I$ is bounded above for all $b \in \Gamma'$, the ideal $I$ has a maximal element which we call $x_{p+1}$.  Similarly, the ideal $I - \{x_{p+1}\}$ has a maximal element $x_{p+2}$.  Iterate this reasoning to produce an infinite sequence $x_{p+1}, x_{p+2},\dots$ of elements.  If $F$ is nonempty, create a linear extension $x_1 \to \cdots \to x_p$ of the order dual $F^*$.  Whether or not $F$ is empty, this produces an infinite sequence $x_1, x_2, \dots$ in which $F_k := \{x_1,\dots,x_k\}$ is a filter of $P$ with minimal element $x_k$ for every $k \ge 1$.  We also have $\kappa(x_j) \in \Gamma'$ for all $j > p$.

For every $k \ge 1$ define the ideal $I_k := P - F_k$.  Then $P_a \cap I_k$ is empty or bounded above for all $a \in \Gamma$.  If $b \in \Gamma'$, then $P_b \cap I_k$ is nonempty and contains a maximal element; call this element $x$.  Then by local finiteness the set $U(x,I_k)$ is finite.  Hence $U_b(I_k)$ is defined for all $b \in \Gamma'$ and all $k \ge 1$.

Set $m := |\Gamma'|$ and fix an ordering $a_1,a_2,\dots,a_m$ of the colors in $\Gamma'$.  For every $k \ge 1$ form the $m$-tuple $\textbf{U}_k := (U_{a_1}(I_k),U_{a_2}(I_k),\dots,U_{a_m}(I_k))$.  By Lemma \ref{LemUpperFrontier}(a)(i) we see $\textbf{U}_k \subseteq \{0,1,2\}^m$, so there are $3^{m}$ possibilities for $\textbf{U}_k$.  Hence we may fix $q > 1$ such that $\textbf{U}_q = \textbf{U}_j$ for infinitely many $j \ge 1$.

Now fix some $r > p+q$ such that $\textbf{U}_q = \textbf{U}_r$.  Since $r - q > p$, we see $\{x_r,\dots,x_{r-q}\} \subseteq I_p = I$.  Hence the elements $x_r,\dots,x_{r-q}$ all have colors in $\Gamma'$.  These elements are respectively minimal in the filters $F_r,\dots,F_{r-q}$ corresponding to the ideals $I_r,\dots,I_{r-q}$.  Thus $U_{\kappa(x_r)}(I_r) = \cdots = U_{\kappa(x_{r-p})}(I_{r-q}) = 2$ by Lemma \ref{LemUpperFrontier}(a)(ii).

We construct a new sequence of ideals of $P$ starting with the ideal $K_q := I_q$.  Let $J_q := P - K_q$ be its corresponding filter.  Since $\kappa(x_r) \in \Gamma'$ and $(U_{a_1}(K_q),\dots,U_{a_m}(K_q)) = \textbf{U}_q = \textbf{U}_r$, we know $U_{\kappa(x_r)}(K_q) = 2$.  By Lemma \ref{LemUpperFrontier}(a)(ii) the filter $J_q$ has a minimal element $y_q$ of color $\kappa(x_r)$.  Define the ideal $K_{q-1} := K_q \cup \{y_q\}$ and let $J_{q-1} := P - K_{q-1}$ be its corresponding filter.  Note by Remark \ref{RemUpperandLowerSumAdditionalElement} that for every $b \in \Gamma'$ the upper frontier census $U_b(K_{q-1})$ is defined and $U_b(K_{q-1}) = U_b(K_q) - \theta_{\kappa(y_q),b} = U_b(I_r) - \theta_{\kappa(x_r),b} = U_b(I_{r-1})$.  Hence $(U_{a_1}(K_{q-1}),\dots,U_{a_m}(K_{q-1})) = \textbf{U}_{r-1}$.  Since $\kappa(x_{r-1}) \in \Gamma'$ we see $U_{\kappa(x_{r-1})}(K_{q-1}) = U_{\kappa(x_{r-1})}(I_{r-1}) = 2$.  Hence the filter $J_{q-1}$ has a minimal element $y_{q-1}$ of color $\kappa(x_{r-1})$ by Lemma \ref{LemUpperFrontier}(a)(ii).  Define the ideal $K_{q-2} := K_{q-1} \cup \{y_{q-1}\}$ and its corresponding filter $J_{q-2} := P-K_{q-2}$.

This reasoning and construction can be repeated to create a sequence of ideals $K_q \subseteq \cdots \subseteq K_0$ where $|K_j - K_{j+1}| = 1$ for every $0 \le j \le q-1$ and $(U_{a_1}(K_0),\dots,U_{a_m}(K_0)) = \textbf{U}_{r-q}$.  Since $\kappa(x_{r-q}) \in \Gamma'$, we also have $U_{\kappa(x_{r-q})}(K_0) = U_{\kappa(x_{r-q})}(I_{r-q}) = 2$.  This implies by Lemma \ref{LemUpperFrontier}(a)(ii) that the filter $J_0 := P - K_0$ contains a minimal element of color $\kappa(x_{r-q})$.  However, since there were only $q$ elements in $J_q = F_q$, we must have $K_0 = P$ and so the filter $J_0$ is empty.
\end{proof}

To prove Proposition \ref{PropLCB2}, we need Part (b) of the following lemma.

\begin{lemma}\label{LemPConnectedImpliesGammaConnected}
Let $P$ be a locally finite poset.
\begin{enumerate}[(a),nosep]
    \item If $P$ is connected, then any two elements can be joined by a finite path in the Hasse diagram of $P$.
    \item Suppose $P$ is a $\Gamma$-colored poset that satisfies NA.  If $P$ is connected, then $\Gamma$ is connected.
\end{enumerate}
\end{lemma}

\begin{proof}
Suppose there is no finite path between elements $x$ and $y$ in the Hasse diagram of $P$.  Let $Q$ be the set of elements that can be connected to $x$ via a finite path in the Hasse diagram, and let $R := P - Q$.  
Since $Q$ and $R$ respectively contain $x$ and $y$, they are nonempty.
Let $u \le v$ in $P$.  Since $P$ is locally finite, there is a finite path from $u$ to $v$ in $P$.  Since two finite paths concatenate into a finite path, we see $u$ is in $Q$ if and only if $v$ is in $Q$.  Hence $P$ is the disjoint union of its subposets $Q$ and $R$, so $P$ is not connected.  This proves (a).

For (b), let $a,b \in \Gamma$.  Since $\kappa$ is surjective, there are elements $x$ and $y$ in $P$ with respective colors $a$ and $b$.  Use (a) to produce a finite path in the Hasse diagram of $P$ from $x$ to $y$.  This path induces a path in $\Gamma$ from $a$ to $b$ by NA.  Thus $\Gamma$ is connected and (b) holds.
\end{proof}

The second main result of this section obtains a lower frontier census bound for connected $\Gamma$-colored $d$-complete posets when the color chains are uniformly unbounded above.

\begin{proposition}\label{PropLCB2}
Suppose that $P$ is a connected $\Gamma$-colored $d$-complete poset and for every $b \in \Gamma$, the set $P_b$ is unbounded above.  Then $P$ satisfies LCB2.
\end{proposition}

\begin{proof}
Suppose for a contradiction that $P$ fails LCB2.  Then there is a color $a \in \Gamma$, an element $x$ minimal in $P_a$, and a finite set $S \subseteq L(x,P)$ with $\sum_{z \in S} -\theta_{\kappa(z),a} > 2$.  Let $F_0$ be the filter generated by the elements in $S$ and let $I_0 := P-F_0$ be its corresponding ideal.  Note that $L_a(F_0)$ is defined and $L_a(F_0) > 2$.  The filter $F_0$ is infinite since it contains the infinite set $P_a$, and it has a minimal element $x_1$.  Since $P_b \cap (F_0-\{x_1\})$ is empty or bounded below for all $b \in \Gamma$ by local finiteness, the filter $F_0 - \{x_1\}$ has a minimal element $x_2$.  Iterate this reasoning to produce an infinite sequence $x_1, x_2,\dots$ in which $F_k := F_0 - \{x_1,\dots,x_k\}$ is a filter of $P$ and $x_k$ is maximal in the corresponding ideal $I_k := P - F_k$ for all $k \ge 1$.

Since the sequence $x_1, x_2,\dots$ is infinite but $\Gamma$ is finite, there is at least one color $c \in \Gamma$ such that there are infinitely many positive integers $j$ with $\kappa(x_j) = c$.  Suppose $d \in \Gamma$ is adjacent to $c$.  By EC and AC the set $P_c \cup P_d$ is a chain.  Since both $P_c$ and $P_d$ are unbounded above, by local finiteness every element of color $c$ (respectively $d$) must have an element of color $d$ (respectively $c$) above it.  Thus no more than two elements of color $c$ can appear consecutively moving up in $P_c \cup P_d$ without an element of color $d$ appearing by ICE2.  This statement also holds for $x_1,x_2,\dots$ since elements moving up in $P_c \cup P_d$ cannot be skipped over when constructing the sequence.  Hence there are infinitely many positive integers $j$ with $\kappa(x_j) = d$.  Since $\Gamma$ is connected by Lemma \ref{LemPConnectedImpliesGammaConnected}(b), it follows that there are infinitely many elements of each color in $x_1, x_2, \dots$; thus $k \ge 0$ can always be chosen large enough so that $I_k$ contains arbitrarily many elements of each color.  


For every $b \in \Gamma$ and $k \ge 0$, the elements $x_{k+1},x_{k+2},\dots$ are in $F_k$ and contain infinitely many elements of color $b$.  Hence the set $P_b \cap F_k$ is nonempty and contains a minimal element.  If $y$ is this element, then $L(y,F_k)$ is finite by local finiteness.  Hence $L_b(F_k)$ is defined for every $b \in \Gamma$ and $k \ge 0$.  

Set $m := |\Gamma|$ and fix an ordering $a_1,a_2,\dots,a_m$ of the colors in $\Gamma$.  For every $k \ge 0$ form the $m$-tuple $\textbf{L}_k := (L_{a_1}(F_k),L_{a_2}(F_k),\dots,L_{a_m}(F_k))$.  Fix $p > 0$ such that $I_p$ contains at least one element of each color.  Note by Lemma \ref{LemUpperFrontier}(b)(i) that $L_b(F_k) \in \{0,1,2\}$ for every $b \in \Gamma$ and $k \ge p$.  Thus there are $3^m$ possibilities for $\textbf{L}_k$ when $k \ge p$, so we may fix $q > p$ such that $\textbf{L}_q = \textbf{L}_j$ for infinitely many $j \ge 0$.  The elements $x_1,\dots,x_q$ are respectively maximal in the ideals $I_1,\dots,I_q$ corresponding to the filters $F_1,\dots,F_q$.  Thus $L_{\kappa(x_1)}(F_1) = \cdots = L_{\kappa(x_q)}(F_q) = 2$ by Lemma \ref{LemUpperFrontier}(b)(ii).


Now fix some $r > q$ such that $\textbf{L}_r = \textbf{L}_q$ and such that $I_r$ contains at least $q+1$ elements of each color.  We construct a new sequence of filters of $P$ starting with the filter $J_r := F_r$.  Let $K_r := P - J_r$ be its corresponding ideal.  Since $(L_{a_1}(J_r),\dots,L_{a_m}(J_r)) = \textbf{L}_r = \textbf{L}_q$, we know that $L_{\kappa(x_q)}(J_r) = 2$.  By Lemma \ref{LemUpperFrontier}(b)(ii) the ideal $K_r$ has a maximal element $y_r$ of color $\kappa(x_q)$.  Define the filter $J_{r-1} := J_r \cup \{y_r\}$ and let $K_{r-1} := P - J_{r-1}$ be its corresponding ideal.  Note by Remark \ref{RemUpperandLowerSumAdditionalElement} that for all $b \in \Gamma$ the lower frontier census $L_b(J_{r-1})$ is defined and $L_b(J_{r-1}) = L_b(J_r) - \theta_{\kappa(y_r),b} = L_b(F_q) - \theta_{\kappa(x_q),b} = L_b(F_{q-1})$.  Hence $(L_{a_1}(J_{r-1}),\dots,L_{a_m}(J_{r-1})) = \textbf{L}_{q-1}$.  Since $K_{r-1}$ contains at least $q$ elements of each color and $L_{\kappa(x_{q-1})}(J_{r-1}) = L_{\kappa(x_{q-1})}(F_{q-1}) = 2$, the ideal $K_{r-1}$ has a maximal element $y_{r-1}$ of color $\kappa(x_{q-1})$ by Lemma \ref{LemUpperFrontier}(b)(ii).  Define the filter $J_{r-2} := J_{r-1} \cup \{y_{r-1}\}$ and its corresponding ideal $K_{r-2} := P - J_{r-2}$.

This reasoning and construction can be repeated to create a sequence of filters $J_r \subseteq \cdots \subseteq J_{r-q}$ where $|J_{r-l} - J_{r-l+1}| = 1$ for every $1 \le l \le q$ and $(L_{a_1}(J_{r-q}),\dots,L_{a_m}(J_{r-q})) = \textbf{L}_0$.  We are always permitted to use Lemma \ref{LemUpperFrontier}(b)(ii) in this construction since the ideal $K_{r-l} = P - J_{r-l}$ contains at least $q + 1 - l$ elements of each color for $1 \le l \le q$.  
Since $K_{r-q}$ contains at least one element of each color, Lemma \ref{LemUpperFrontier}(b)(i) shows $L_b(J_{r-q}) \in \{0,1,2\}$ for all $b \in \Gamma$.  This contradicts $L_a(F_0) > 2$.
\end{proof}


\section{Connected infinite \texorpdfstring{$\Gamma$}{Gamma}-colored \texorpdfstring{$d$}{d}-complete posets}\label{MainProof}

Theorem \ref{TheoremClassifyInfinite} is our main result of this section, which shows that connected infinite $\Gamma$-colored $d$-complete posets are precisely filters of connected full heaps.  We show this by examining cases resulting from Lemma \ref{LemBoundedUniformly}.
In the most involved case, we apply the results of Section \ref{SectionTwoImportantPropositions} to construct a full heap containing a given $\Gamma$-colored $d$-complete poset as a filter.

In this section, we refer to the following properties $P$ may satisfy:
\begin{itemize}[nosep]
    \item [] (CBA) For every $b \in \Gamma$, the color set $P_b$ is bounded above.
    \item [] (CUA) For every $b \in \Gamma$, the color set $P_b$ is unbounded above.
\end{itemize}
\noindent We start with the lemma that will establish cases to consider.

\begin{lemma}\label{LemBoundedUniformly}
Suppose that $P$ is a connected $\Gamma$-colored poset that satisfies EC, NA, AC, and UCB$k$ for some $k \ge 1$.  Then $P$ satisfies CBA or CUA.
\end{lemma}

\begin{proof}
Suppose that $P$ fails CUA.  Then there is a color $b$ such that $P_b$ is bounded above.  Let $x$ be maximal in $P_b$ and let $c \in \Gamma$ be adjacent to $b$.  By EC and AC the set $P_b \cup P_c$ is a chain.  By local finiteness the portion of $P_c$ less than $x$ must be empty or have a maximal element.  The portion of $P_c$ greater than $x$ is in $U(x,P)$ and must consist of at most $k$ elements by UCB$k$.  Hence $P_c$ is bounded above.  Since $\Gamma$ is connected by Lemma \ref{LemPConnectedImpliesGammaConnected}(b), we see $P_a$ is bounded above for all $a \in \Gamma$.  Thus $P$ satisfies CBA.
\end{proof}

Suppose $P$ is a connected infinite $\Gamma$-colored $d$-complete poset.  By Proposition \ref{PropProctorProof} we know that $P$ does not satisfy CBA.  Hence by Lemma \ref{LemBoundedUniformly} it must satisfy CUA.  Then Proposition \ref{PropLCB2} shows $P$ satisfies LCB2.  A dualized version of Lemma \ref{LemBoundedUniformly} shows that $P$ satisfies one of the following two properties.
\begin{itemize}[nosep]
    \item [] (CBB) For every $b \in \Gamma$, the color set $P_b$ is bounded below.
    \item [] (CUB) For every $b \in \Gamma$, the color set $P_b$ is unbounded below.
\end{itemize}
Thus there are two cases: either $P$ satisfies CUA and CUB, or it satisfies CUA and CBB.  The first case is straightforward.

\begin{proposition}\label{PropUnboundedDoublyImpliesFullHeap}
Suppose that $P$ is a connected infinite $\Gamma$-colored $d$-complete poset.  If $P$ satisfies CUA and CUB, then $P$ is a full heap.
\end{proposition}

\begin{proof}
Since $P$ satisfies CUA and CUB and is locally finite, the subposet $P_a$ is isomorphic to $\mathbb{Z}$ for every $a \in \Gamma$.  Hence $P$ satisfies G3.  By Corollary \ref{CorFullHeapEquiv} we see $P$ is a full heap.
\end{proof}

We now focus on the case where $P$ satisfies CUA and CBB, which culminates in Corollary \ref{CorFilterOfFullHeap}.

\begin{lemma}\label{LemUnboundedAboveOnlyNoLCB1}
Suppose that $P$ is a connected infinite $\Gamma$-colored $d$-complete poset.  If $P$ satisfies CUA and CBB, then it does not satisfy LCB1. 
\end{lemma}

\begin{proof}
Let $P^*$ be the order dual poset of $P$, and note it satisfies EC, NA, AC, and ICE2.  Since $P$ satisfies CBB, we see $P^*$ satisfies CBA.  Since $P^*$ is infinite, we know by Proposition \ref{PropProctorProof} that it is not $\Gamma$-colored $d$-complete.  Hence $P^*$ does not satisfy UCB1, and so 
$P$ does not satisfy LCB1.
\end{proof}


The notation established in the following discussion will be maintained through the proof of Lemma \ref{LemQFullHeap}.  Let $P$ be a connected infinite $\Gamma$-colored $d$-complete poset that satisfies CUA and CBB.  We describe an iterative process which will produce a new poset $Q$ by ``extending'' $P$ downwardly in a manner similar to the classifications of finite $d$-complete posets in \cite{DDCT} and finite $\Gamma$-colored minuscule posets in \cite{PaperClassifyMinuscule}.  The poset $P$ will be a filter of $Q$, and we will show $Q$ is a full heap in Lemma \ref{LemQFullHeap}.  

By Proposition \ref{PropLCB2} we know that $P$ satisfies LCB2.  By Lemma \ref{LemUnboundedAboveOnlyNoLCB1} we know $P$ does not satisfy LCB1.  Thus we may choose a color $a \in \Gamma$ for which the minimal element $y$ of $P_a$ satisfies $\sum_{z \in L(y,P)} -\theta_{\kappa(z),a} = 2$.  We note that $a$ may not be the only color with this property.  There are two possibilities:
\begin{enumerate}[(i),nosep]
    \item The set $L(y,P)$ contains exactly two elements $u$ and $v$.  Their colors are both 1-adjacent to $a$.
    \item The set $L(y,P)$ contains exactly one element $w$.  Its color is 2-adjacent to $a$.
\end{enumerate}
Create the symbol $x$ and give it the color $a$.  Set $Q_1 := P \cup \{x\}$.  Define a partial order on $Q_1$ 
as the reflexive transitive closure of the order on $P$ and the following new covering relation(s),
depending on whether possibility (i) or (ii) listed above applies:
\begin{enumerate}[(i),nosep]
    \item If $u < v$ (respectively $v < u$), then add the covering relation $x \to u$ (respectively $x \to v$).  If $u$ and $v$ are incomparable, then add the covering relations $x \to u$ and $x \to v$.
    \item Add the covering relation $x \to w$.
\end{enumerate}

\begin{lemma}\label{LemQ1Properties}
The poset $Q_1$ is a connected infinite $\Gamma$-colored $d$-complete poset that satisfies CUA and CBB.  The poset $P$ is a filter of $Q_1$.
\end{lemma}

\begin{proof}
It is immediate that $Q_1$ is connected, locally finite, satisfies CUA and CBB, and contains $P$ as a filter.  It satisfies UCB1 vacuously.  The properties EC, NA, and AC follow directly from the construction of $Q_1$.  The only pair of consecutive elements of the same color in $Q_1$ that is not contained in $P$ is $x < y$.  By construction, the interval $(x,y)$ contains $L(y,P)$ and no other elements with colors adjacent to $a$.  Thus $\sum_{z \in (x,y)} -\theta_{\kappa(z),a} = \sum_{z \in L(y,P)} -\theta_{\kappa(z),a} = 2$, so ICE2 is satisfied.
\end{proof}

Thus Proposition \ref{PropLCB2} and Lemma \ref{LemUnboundedAboveOnlyNoLCB1} also apply to $Q_1$.  The construction that led to $Q_1$ can be iterated to construct an infinite sequence of posets $Q_0 := P,Q_1,Q_2,Q_3,\dots$, all of which are connected infinite $\Gamma$-colored $d$-complete posets satisfying CUA and CBB.  
These posets are constructed via a sequence $x_1 := x, x_2, x_3, \dots$ of elements, where $Q_i := Q_{i-1} \cup \{x_i\}$ and where $x_i$ is minimal in $Q_i$ for all $i \ge 1$.  Set $Q := P \cup \left( \bigcup_{i=1}^\infty \{x_i\} \right)$.  Define a partial order on $Q$ as the reflective transitive closure of the order on $P$ and the covering relations added to form each poset $Q_i$ for $i \ge 1$.  Then $Q$ is a $\Gamma$-colored poset.

\begin{lemma}\label{LemQFullHeap}
The poset $Q$ is a connected full heap and $P$ is a filter of $Q$.
\end{lemma}

\begin{proof}
The posets $P = Q_0,Q_1,Q_2,\dots$ are all filters of $Q$ by construction.  Let $x,y \in Q$ with $x < y$.  Then there exists $k \ge 0$ with $x < y$ in $Q_k$.  Since $Q_k$ is a filter of $Q$, we have $[x,y] \subseteq Q_k$.  Since $Q_k$ is locally finite, we see $[x,y]$ is finite.  Hence $Q$ is locally finite.

Suppose for a contradiction that $Q$ can be written as the disjoint union of two nonempty subposets $R$ and $S$.  Let $u \in R$ and $v \in S$.  There exists some $k \ge 0$ such that $u,v \in Q_k$.  Then $Q_k$ is the disjoint union of two nonempty subposets $Q_k \cap R$ and $Q_k \cap S$, violating that $Q_k$ is connected.  Hence $Q$ is connected.

The poset $Q$ satisfies EC, NA, AC, and ICE2; any occurrence of neighbors or two elements with equal or adjacent colors in $Q$ occurs in $Q_k$ for some $k \ge 0$, and $Q_k$ satisfies these properties.  
Suppose $\bigcup_{i=1}^\infty \{x_i\}$ contains finitely many elements (possibly zero) of some color.  Since $\Gamma$ is connected by Lemma \ref{LemPConnectedImpliesGammaConnected}(b), we can find adjacent colors $a,b \in \Gamma$ such that $\bigcup_{i=1}^\infty \{x_i\}$ contains finitely many (possibly zero) elements of color $a$ and infinitely many elements of color $b$.  Then $Q_a := \{z \in Q \ | \ \kappa(z) = a\}$ is bounded below in $Q$; let $x$ be minimal in $Q_a$.  By AC we know $x$ is comparable to every element in $Q_b := \{z \in Q \ | \ \kappa(z) = b\}$.  By local finiteness, since $Q_b$ is unbounded below there must be at least one element of color $b$ below $x$.  Thus there are infinitely many such elements.  Then $Q_j$ fails LCB2 for some $j \ge 1$, which violates Proposition \ref{PropLCB2} applied to $Q_j$.  Hence $\bigcup_{i=1}^\infty \{x_i\}$ contains infinitely many elements of each color, and so $Q$ satisfies both CUA and CUB.  Hence $Q$ satisfies G3, so by Corollary \ref{CorFullHeapEquiv} we see $Q$ is a full heap.
\end{proof}

The preceding work resolves the case in which $P$ satisfies CUA and CBB.

\begin{corollary}\label{CorFilterOfFullHeap}
Suppose that $P$ is a connected infinite $\Gamma$-colored $d$-complete poset.  If $P$ satisfies CUA and CBB, then $P$ is a filter of some connected full heap.
\end{corollary}


We now obtain the main result of this section.

\begin{theorem}\label{TheoremClassifyInfinite}

Let $P$ be a connected infinite $\Gamma$-colored poset.  Then $P$ is a $\Gamma$-colored $d$-complete poset if and only if it is a filter of some connected full heap.
\end{theorem}

\begin{proof}
Let $P$ be a connected infinite $\Gamma$-colored $d$-complete poset.  Lemma \ref{LemBoundedUniformly} shows $P$ satisfies CBA or CUA.  Since $P$ is infinite, by Proposition \ref{PropProctorProof} it must satisfy CUA.  Then Proposition \ref{PropLCB2} shows $P$ satisfies LCB2.  The dualized version of Lemma \ref{LemBoundedUniformly} shows $P$ satisfies CBB or CUB.  If $P$ satisfies CUB, then Proposition \ref{PropUnboundedDoublyImpliesFullHeap} shows $P$ is a full heap (and it is a filter of itself).  If $P$ satisfies CBB, then Corollary \ref{CorFilterOfFullHeap} shows that $P$ is the filter of some connected full heap.

Now let $P$ be the filter of some connected full heap.  By Corollary \ref{CorFullHeapEquiv} this full heap has the properties EC, NA, AC, and ICE2; these properties are preserved in filters.  By G3 the set $P_a$ is unbounded above for all $a \in \Gamma$, so UCB1 is vacuously satisfied.  Thus $P$ is $\Gamma$-colored $d$-complete.
\end{proof}



\section{Classification of \texorpdfstring{$\Gamma$}{Gamma}-colored \texorpdfstring{$d$}{d}-complete posets}\label{SectionClassifydComplete}

We classify all $\Gamma$-colored $d$-complete posets in this section.  To do this, Proposition \ref{PropConnected} shows that it is sufficient to classify the connected $\Gamma$-colored $d$-complete posets.  Theorems \ref{TheoremMSequivJRS} and \ref{TheoremClassifyInfinite} allow us to apply the classification of dominant minuscule heaps by R.A. Proctor and J.R. Stembridge and the classification of full heaps by R.M. Green and Z.S. McGregor-Dorsey.  After describing these previous classifications, we obtain our main result in Theorem \ref{TheoremClassify}, namely, the classification of $\Gamma$-colored $d$-complete posets.


\begin{lemma}\label{LemHeapConnected}
Suppose that $P$ is $\Gamma$-colored and satisfies EC.
\begin{enumerate}[(a),nosep]
    \item Then $P$ is a disjoint union of at most $|\Gamma|$ connected posets; call them $P_1,\dots,P_r$ for some $1 \le r \le |\Gamma|$.
    \item Suppose $P$ also satisfies NA and AC.  Then $\Gamma$ is the disjoint union of $r$ connected Dynkin diagrams $\Gamma_1,\dots,\Gamma_r$, and $P_k$ is $\Gamma_k$-colored for $1 \le k \le r$.
\end{enumerate}
\end{lemma}

\begin{proof}
Any collection of elements taken from different connected components of $P$ form an antichain in $P$.  Their colors are distinct by EC.  Therefore the number of connected components of $P$ is at most $|\Gamma|$, which proves (a).  Let $P_1,\dots,P_r$ be the connected components of $P$ for some $1 \le r \le |\Gamma|$.

Now suppose $P$ additionally satisfies NA and AC.  For each $1 \le k \le r$, define $\Gamma_k := \kappa(P_k)$.  
These sets are nonempty, and since $\kappa$ is surjective we have $\Gamma = \bigcup_{k=1}^r \Gamma_k$ as sets of colors.
Fix $1 \le i,j \le r$ with $i \ne j$.  By EC (respectively AC) no element in $P_i$ can have a color that is the same as (respectively adjacent to) the color of an element in $P_j$.  Thus $\Gamma_1,\dots,\Gamma_r$ is a partition of $\Gamma$ and there are no edges in $\Gamma$ between colors in different sets of this partition.
Thus $\Gamma_1,\dots,\Gamma_r$ are Dynkin diagrams using the edges and integer labels from $\Gamma$, and they respectively color $P_1,\dots,P_r$.
These connected components inherit NA from $P$, so by Lemma \ref{LemPConnectedImpliesGammaConnected}(b) their Dynkin diagrams are connected.  This proves (b).
\end{proof}

\noindent So when $P$ satisfies EC, NA, and AC, it must have the same number of connected components as $\Gamma$.  We isolate the most useful version of this result in a corollary that extends Lemma \ref{LemPConnectedImpliesGammaConnected}(b).

\begin{corollary}\label{CorPConnectediffGammaConnected}
Suppose that $P$ satisfies EC, NA, and AC.  Then $P$ is connected if and only if $\Gamma$ is connected.
\end{corollary}


The decomposition in Lemma \ref{LemHeapConnected}(b) extends to $\Gamma$-colored $d$-complete and $\Gamma$-colored minuscule posets.

\begin{proposition}\label{PropConnected}
Let $P$ be a $\Gamma$-colored poset.  Then $P$ is $\Gamma$-colored $d$-complete (respectively $\Gamma$-colored minuscule) if and only if there is some integer $r$ with $1 \le r \le |\Gamma|$, connected posets $P_1,\dots,P_r$, and connected Dynkin diagrams $\Gamma_1,\dots,\Gamma_r$ such that $P$ is the disjoint union of $P_1,\dots,P_r$ and $\Gamma$ is the disjoint union of $\Gamma_1,\dots,\Gamma_r$ and $P_k$ is $\Gamma_k$-colored $d$-complete (respectively $\Gamma_k$-colored minuscule) for all $1 \le k \le r$.
\end{proposition}

\begin{proof}
Suppose $P$ is $\Gamma$-colored $d$-complete (or $\Gamma$-colored minuscule).  Since $P$ satisfies EC, NA, and AC, Lemma \ref{LemHeapConnected}(b) implies the existence of the required disjoint union and coloring decompositions of $P$ and $\Gamma$.  The coloring properties EC, NA, AC, ICE2, UCB1, and LCB1 are stated in terms of comparable elements in $P$ and equal or adjacent colors in $\Gamma$, so each property is preserved under these coloring decompositions.  Hence $P_k$ is $\Gamma_k$-colored $d$-complete (or $\Gamma_k$-colored minuscule) for every $1 \le k \le r$.  Since these coloring properties are preserved under unions, the converse holds.
\end{proof}

\noindent So classifying all $\Gamma$-colored $d$-complete (or minuscule) posets reduces to classifying the connected ones.

The finite $\Gamma$-colored $d$-complete posets are precisely the dominant minuscule heaps by Theorem \ref{TheoremMSequivJRS}.  So to classify the connected ones, we apply the classification of connected dominant minuscule heaps presented in Section 4 of \cite{Ste}.  Here we give a brief overview of this classification.

Let $P$ be a connected dominant minuscule heap colored by a Dynkin diagram $\Gamma$.  Let $T$ be the set of maximal elements of each color and call it the \emph{top tree} of $P$.  Suppose $x,y \in T$ and $x \to y$.  If $y$ is the only element of its color in $P$, then this edge is a \emph{slant edge} of $P$.  The poset $P$ is \emph{slant irreducible} if it contains no slant edges.

Every dominant minuscule heap can be deconstructed into one or more slant irreducible dominant minuscule heaps by removing all slant edges.  These subposets are colored by the Dynkin diagrams produced by removing the corresponding edges from $\Gamma$ that exist by NA.  Conversely, every connected dominant minuscule heap can be built up from slant irreducible dominant minuscule heaps in the following way.  Covering relations are added in which the maximal element of one dominant minuscule heap is covered by an element in another dominant minuscule heap whose color only appears once in that poset. This creates a slant edge in the new poset, and this process is repeated until a connected poset is produced. 
For each newly created slant edge $x \to y$ in the poset, an adjacency is added between the respective colors in the two Dynkin diagrams so that NA and UCB1 are satisfied.  The property NA requires that $\kappa(x) \sim \kappa(y)$, and the property UCB1 requires that $\kappa(y)$ is 1-adjacent to $\kappa(x)$.
This process produces a single connected acyclic Dynkin diagram.
The resulting connected dominant minuscule heap is a \emph{slant sum} of the original slant irreducible dominant minuscule heaps.  Hence the problem of classifying the dominant minuscule heaps is reduced to classifying the slant irreducible ones.

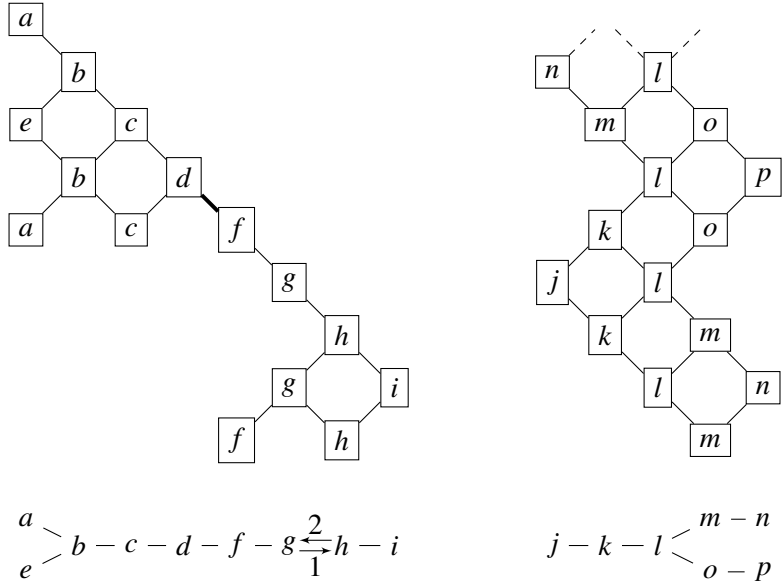
\begin{figure}[t!]
    \centering
    \begin{tikzpicture}[scale=.7]
    \tikzset{edge/.style = {->,> = latex'}}
        \node[draw] (A1) at (0,20){$a$};
        \node[draw] (B1) at (1,19){$b$};
        \node[draw] (C1) at (2,18){$c$};
        \node[draw] (D1) at (3,17){$d$};
        \node[draw] (E1) at (0,18){$e$};
        \node[draw] (B2) at (1,17){$b$};
        \node[draw] (C2) at (2,16){$c$};
        \node[draw] (A2) at (0,16){$a$};
        \node[draw] (F1) at (4,16){$f$};
        \node[draw] (G1) at (5,15){$g$};
        \node[draw] (H1) at (6,14){$h$};
        \node[draw] (I1) at (7,13){$i$};
        \node[draw] (G2) at (5,13){$g$};
        \node[draw] (H2) at (6,12){$h$};
        \node[draw] (F2) at (4,12){$f$};
        
        \draw (A1) -- (B1) -- (C1) -- (D1) -- (C2) -- (B2) -- (C1) (B1) -- (E1) -- (B2) -- (A2) (F1) -- (G1) -- (H1) -- (I1) -- (H2) -- (G2) (H1) -- (G2) -- (F2);
        \draw[line width=1.5pt] (D1) -- (F1);
        
        \node (A) at (0,10.5){$a$};
        \node (B) at (1,10){$b$};
        \node (E) at (0,9.5){$e$};
        \node (C) at (2,10){$c$};
        \node (D) at (3,10){$d$};
        \node (F) at (4,10){$f$};
        \node (G) at (5,10){$g$};
        \node (H) at (6,10){$h$};
        \node (I) at (7,10){$i$};
        \node (GT) at (5,10.1){};
        \node (GB) at (5,9.9){};
        \node (HT) at (6,10.1){};
        \node (HB) at (6,9.9){};
        
        \draw (A) -- (B) -- (E) (B) -- (C) -- (D) -- (F) -- (G) (H) -- (I);
        \draw [edge] (HT) to (GT);
        \draw [edge] (GB) to (HB);
        \node at (5.5,10.4){2};
        \node at (5.5,9.6){1};
        
        \node (J) at (10,10){$j$};
        \node (K) at (11,10){$k$};
        \node (L) at (12,10){$l$};
        \node (M) at (13,10.5){$m$};
        \node (N) at (14,10.5){$n$};
        \node (O) at (13,9.5){$o$};
        \node (P) at (14,9.5){$p$};
        
        \draw (J) -- (K) -- (L) -- (M) -- (N) (L) -- (O) -- (P);
        
        \node[draw] (M1) at (13,12){$m$};
        \node[draw] (N1) at (14,13){$n$};
        \node[draw] (L1) at (12,13){$l$};
        \node[draw] (K1) at (11,14){$k$};
        \node[draw] (J1) at (10,15){$j$};
        \node[draw] (M2) at (13,14){$m$};
        \node[draw] (L2) at (12,15){$l$};
        \node[draw] (K2) at (11,16){$k$};
        \node[draw] (O1) at (13,16){$o$};
        \node[draw] (P1) at (14,17){$p$};
        \node[draw] (L3) at (12,17){$l$};
        \node[draw] (O2) at (13,18){$o$};
        \node[draw] (M3) at (11,18){$m$};
        \node[draw] (N2) at (10,19){$n$};
        \node[draw] (L4) at (12,19){$l$};
        \node (M4) at (11,20){};
        \node (K3) at (13,20){};
        
        \draw[dashed] (N2) -- (M4) -- (L4) -- (K3);
        \draw (M1) -- (L1) -- (K1) -- (J1) -- (K2) -- (L2) -- (K1) (L2) -- (M2) -- (L1) (M2) -- (N1) -- (M1) (K2) -- (L3) -- (O1) -- (L2) (O1) -- (P1) -- (O2) -- (L3) (N2) -- (M3) -- (L3) (O2) -- (L4) -- (M3);
    \end{tikzpicture}
    \caption{A $\Gamma$-colored $d$-complete poset $P$.  Here both $P$ and $\Gamma$ have two connected components.  The connected component on the left is a slant sum of two slant irreducible dominant minuscule heaps (the slant edge is bold).  The connected component on the right is an order filter of a full heap.}
    \label{FigdComplete}
\end{figure}

The classification of slant irreducible dominant minuscule heaps was accomplished by Proctor and Stembridge.  Proctor introduced colored and uncolored $d$-complete posets in \cite{Wave} and showed there is a unique colored $d$-complete poset for each uncolored one; see \cite[Prop. 8.6]{Wave} and note that paper uses an order dualized definition of $d$-complete.  Colored $d$-complete posets are equivalent to dominant minuscule heaps that have been colored by simply laced Dynkin diagrams.  This equivalence has been known to Proctor for some time (personal communication)
and is used implicitly by Stembridge in \cite{Ste}.  There are equivalent notions of slant irreducibility and slant sum in \cite{DDCT} (see the discussion following Fact 7.5.1 in \cite{Str}).  Proctor classified slant irreducible $d$-complete posets into fifteen families \cite[Thm. 7]{DDCT}.  Stembridge extended this classification \cite[Thm. 4.2]{Ste} with two additional families colored by multiply laced Dynkin diagrams.


The connected infinite $\Gamma$-colored $d$-complete posets are precisely the filters of connected full heaps by Theorem \ref{TheoremClassifyInfinite}.  So we apply the classification of full heaps which was accomplished by Green and McGregor-Dorsey.
Green provided the list of full heaps colored by Dynkin diagrams of affine Kac--Moody type in Theorem 6.6.2 of \cite{Gre}.  In  his doctoral thesis written under Green's supervision, McGregor-Dorsey showed in Theorem 4.7.1 of \cite{McD} that the connected components of any finite Dynkin diagram coloring a full heap must have affine type.  Since $P$ is connected if and only if $\Gamma$ is connected by Corollary \ref{CorPConnectediffGammaConnected}, the list provided by Green is a complete list of connected full heaps colored by Dynkin diagrams with finitely many nodes.

We thus obtain the classification of all $\Gamma$-colored $d$-complete posets from Proposition \ref{PropConnected} and Theorems \ref{TheoremMSequivJRS} and \ref{TheoremClassifyInfinite} by applying the previous classifications of Proctor, Stembridge, Green, and McGregor-Dorsey.  Figure \ref{FigdComplete} displays a typical $\Gamma$-colored $d$-complete poset with two connected components.



\begin{theorem}\label{TheoremClassify}
Let $P$ be a $\Gamma$-colored poset.  Then $P$ is $\Gamma$-colored $d$-complete if and only if $P$ is the disjoint union of connected posets in which each is either a finite slant sum of slant irreducible dominant minuscule heaps from \cite{DDCT,Ste} or an infinite filter of some connected full heap from \cite{Gre}.  In this case, the Dynkin diagram $\Gamma$ is the disjoint union of the Dynkin diagrams coloring the connected components of $P$.
\end{theorem}



\section{Weyl group elements and upper \texorpdfstring{$P$}{P}-minuscule representations}\label{RepInteractions}

We describe in this section how the above work applies to Kac--Moody theory.
The algebraic terms stated in this section are defined precisely in \cite{Kum}.
Let $\Gamma$ be a Dynkin diagram and fix an ordering of its colors.  Then $[\theta_{ab}]$ is a \emph{generalized Cartan matrix}.  This matrix may be \emph{symmetrizable} or \emph{indecomposable}.  It may be used to construct a \emph{Kac--Moody algebra} $\mathfrak{g}$ with \emph{Cartan subalgebra} $\mathfrak{h}$ and \emph{standard Borel subalgebra} $\mathfrak{b}$.  

The definition of Kac--Moody algebra in \cite{Kum} agrees with the definition in \cite{Kac} when $[\theta_{ab}]$ is symmetrizable.  This is always the case when $P$ is a $\Gamma$-colored $d$-complete poset.



\begin{proposition}\label{PropSymmetrizable}
Let $P$ be a $\Gamma$-colored $d$-complete poset.  Then the corresponding generalized Cartan matrix $[\theta_{ab}]$ is symmetrizable.
\end{proposition}

\begin{proof}
We may assume that $[\theta_{ab}]$ is indecomposable.  Then $\Gamma$ is connected, as is $P$ by Corollary \ref{CorPConnectediffGammaConnected}.
If $P$ is infinite, then it must be a filter of some connected full heap by Theorem \ref{TheoremClassifyInfinite} and hence colored by a Dynkin diagram of affine type.  Every affine generalized Cartan matrix is symmetrizable.

Now assume $P$ is finite and thus a dominant minuscule heap by Theorem \ref{TheoremMSequivJRS}.  Fix $k \ge 1$ and colors $a_1,\dots,a_k \in \Gamma$.  
We show that $\theta_{a_1,a_2} \theta_{a_2,a_3} \cdots \theta_{a_k,a_1} = 
\theta_{a_2,a_1} \theta_{a_3,a_2} \cdots \theta_{a_1,a_k}$.  
Note $\theta_{ab} = 0$ if and only if $\theta_{ba} = 0$ and $\theta_{ab} = 2$ if and only if $\theta_{ba} = 2$ for all $a,b \in \Gamma$, so assume no such terms appear in either product.  Hence $\theta_{a_1,a_2} \theta_{a_2,a_3} \cdots \theta_{a_k,a_1}$ is the product formed by multiplying generalized Cartan matrix entries along the path $a_1,a_2,\dots,a_k,a_1$ starting and ending at $a_1$ in $\Gamma$.  Since $\Gamma$ is acyclic by S4, an edge from $a$ to $b$ is traversed on this path (with multiplicity) if and only if the reverse edge from $b$ to $a$ is traversed (with the same multiplicity).  Hence $\theta_{ab}$ and $\theta_{ba}$ appear the same number of times in $\theta_{a_1,a_2} \theta_{a_2,a_3} \cdots \theta_{a_k,a_1}$ for every $a,b \in \Gamma$, so $\theta_{a_1,a_2} \theta_{a_2,a_3} \cdots \theta_{a_k,a_1} = \theta_{a_2,a_1} \theta_{a_3,a_2} \cdots \theta_{a_1,a_k}$.  Then $[\theta_{ab}]$ is symmetrizable by \cite[Lem. 15.15]{Carter}.
\end{proof}

The \emph{Kac--Moody derived subalgebra} is $\mathfrak{g}' := [\mathfrak{g},\mathfrak{g}]$.  In \cite{Unify} we defined the \emph{Borel derived subalgebra} $\mathfrak{b}' := \mathfrak{b} \cap \mathfrak{g}'$.  This subalgebra is generated by the symbols $\{x_a,h_a\}_{a \in \Gamma}$ subject to the relations


\begin{itemize}[nosep]
    \item [] (XX) \  $\underbrace{[x_a,[x_a,\dots,[x_a}_{\text{$1 - \theta_{ba}$ times}},x_b] \dots ]] = 0$ for all $a,b \in \Gamma$ such that $a \ne b$,
    \item [] (HH) \ $[h_b,h_a] = 0$ for all $a,b \in \Gamma$, and
    \item [] (HX) \ $[h_b,x_a] = \theta_{ab} x_a$ for all $a,b \in \Gamma$.
\end{itemize} 





The elements $\{h_b\}_{b \in \Gamma}$ are in $\mathfrak{h}$ and are the \emph{simple coroots}.  The \emph{simple roots} $\{\alpha_b\}_{b \in \Gamma}$ in the dual space $\mathfrak{h}^*$ satisfy $\alpha_b(h_c) = \theta_{bc}$ for $b,c \in \Gamma$.  Acting on $\mathfrak{h}^*$ is the \emph{Weyl group} $W$ with \emph{simple generators} $\{s_b\}_{b \in \Gamma}$ via the rule $s_b(\lambda) := \lambda - \lambda(h_b)\alpha_b$ for $b \in \Gamma$ and $\lambda \in \mathfrak{h}^*$.  Every element $w \in W$ can be written as a product of simple generators, and such products of minimal \textit{length} are \emph{reduced expressions} for $w$.  An element $\lambda \in \mathfrak{h^*}$ is an \emph{integral weight} if $\lambda(h_b) \in \mathbb{Z}$ for all $b \in \Gamma$ and a \emph{dominant integral weight} if $\lambda(h_b) \in \mathbb{Z}_{\ge 0}$ for all $b \in \Gamma$.  If $\lambda$ is an integral weight, then D. Peterson defined \cite{Car} an element $w \in W$ to be \emph{$\lambda$-minuscule} if $w = s_{b_k} \cdots s_{b_1}$ for some $k \ge 1$ and $s_{b_j}(s_{b_{j-1}} \cdots s_{b_1})(\lambda) = (s_{b_{j-1}} \cdots s_{b_1})(\lambda) - \alpha_{b_j}$ for all $1 \le j \le k$.  

Section 3 of \cite{Ste} describes the equivalence between dominant minuscule heaps (i.e. finite $\Gamma$-colored $d$-complete posets) and $\lambda$-minuscule Weyl group elements for dominant integral weights $\lambda$.  Linear extensions of the former correspond to reduced expressions of the latter; see also \cite[Cor. 5.5]{Wave}.  The following result confirms a claim made by Proctor in Section 15 of \cite{DDCT}.  This claim was made as an application of his classification of $d$-complete posets (the dominant minuscule heaps colored by simply laced Dynkin diagrams).  Our proof uses Proposition \ref{PropProctorProof} instead of the classifications of Proctor and Stembridge.

\begin{proposition}
Let $\lambda$ be a dominant integral weight.  Then there is no infinite sequence $a_1,a_2,\dots$ of colors such that $s_{a_k} \cdots s_{a_1}$ is $\lambda$-minuscule for every $k \ge 1$.
\end{proposition}


\begin{proof}
Assume for a contradiction that $a_1,a_2,\dots$ is an infinite sequence of colors such that $w_k := s_{a_k} \cdots s_{a_1}$ is $\lambda$-minuscule for every $k \ge 1$.  Let $\Gamma' \subseteq \Gamma$ be the subdiagram whose nodes are the colors appearing in this sequence.  Let $j \ge 1$ be so that $w_j$ contains simple reflections corresponding to each color in $\Gamma'$.  For every $k \ge j$, let $P_k$ be the dominant minuscule heap (i.e. finite $\Gamma'$-colored $d$-complete poset) corresponding to $w_k$.  For every $k > j$, the poset $P_k$ can be produced by adding some minimal element $x_k$ to $P_{k-1}$.  Let $P := P_j \cup \left( \bigcup_{k > j} \{x_k\} \right)$ and give $P$ the partial order that is the reflexive transitive closure of the order on $P_j$ and the covering relations resulting from the additions of the elements in $\{x_k\}_{k > j}$.  The poset $P$ satisfies EC, NA, AC, and ICE2; any occurrence of neighbors or elements with equal or adjacent colors in $P$ occurs in $P_k$ for some $k \ge j$, and $P_k$ satisfies these properties.  The poset $P$ also has the same set of maximal elements of each color as $P_j$, so $P$ satisfies UCB1.  Hence $P$ is an infinite $\Gamma'$-colored $d$-complete poset for which $P_b$ is bounded above for all $b \in \Gamma$.  This violates Proposition \ref{PropProctorProof}.
\end{proof}

Let $P$ be a locally finite poset colored by a Dynkin diagram $\Gamma$.  Let $\mathcal{FI}(P)$ be the set of all pairs $(F,I)$, where $F$ is a filter of $P$ and $I := P - F$ is its corresponding ideal.  We call these the \emph{splits} of $P$.  Let $V := \langle \mathcal{FI}(P) \rangle$ be the free complex vector space generated by $\mathcal{FI}(P)$.  Denote the vector corresponding to a split $(F,I)$ as $\langle F,I \rangle$.  Let $(F,I) \in \mathcal{FI}(P)$ and $a \in \Gamma$.  If there are finitely many minimal elements in $F$ of color $a$, then define $X_a. \langle F,I \rangle := \sum \langle F - \{x\}, I \cup \{x\} \rangle$, where the sum is taken over the minimal elements of color $a$ in $F$.  If these actions can be defined for all $a \in \Gamma$ and $(F,I) \in \mathcal{FI}(P)$, then extend the $X_a$ linearly to $V$ for all $a \in \Gamma$.  The operators $\{X_a\}_{a \in \Gamma}$ are the \emph{color raising} operators for $\mathcal{FI}(P)$.  

In \cite{Unify}, we made the following definitions.

\begin{definition}\label{DefCarries}
Let $P$ be a $\Gamma$-colored poset and suppose that the color raising and lowering operators $\{X_a\}_{a \in \Gamma}$ are defined on $V$.  Let $[A,B] := AB - BA$ be the commutator on $\text{End}(V)$.  We say $\mathcal{FI}(P)$ \emph{carries a representation} of $\mathfrak{b}'$ if there exist diagonal operators $\{H_a\}_{a \in \Gamma}$ on $V$ (with respect to the basis of splits) such that $\{X_a,H_a\}_{a \in \Gamma}$ satisfy XX and HX with respect to the commutator on $\text{End}(V)$.
\end{definition}

\noindent Note the relation HH holds automatically since the operators $\{H_a\}_{a \in \Gamma}$ are diagonal on $V$.  So in Definition \ref{DefCarries}, all of the relations for $\mathfrak{b}'$ hold for the operators $\{X_a,H_a\}_{a \in \Gamma}$ with respect to the commutator on $\text{End}(V)$.   Thus the maps $x_a \mapsto X_a$ and $h_a \mapsto H_a$ for all $a \in \Gamma$ induce a Lie algebra homomorphism from $\mathfrak{b}'$ to $\mathfrak{gl}(V)$.  Therefore $V$ is a representation of $\mathfrak{b}'$.  We say this representation is \emph{carried by $\mathcal{FI}(P)$}.

\begin{definition}
A representation of $\mathfrak{b}'$ carried by $\mathcal{FI}(P)$ 
is \emph{upper $P$-minuscule} if 
$X_a^2 = 0$ for all $a \in \Gamma$, 
the eigenvalues of $\{H_a\}_{a \in \Gamma}$ on the basis of splits are in the set $\{-1,0,1,2,\dots\}$, and for all $a \in \Gamma$ we have $H_a.\langle F,I \rangle = -\langle F,I \rangle$ if and only if $F$ has a minimal element of color $a$. 
\end{definition}

Finite dimensional upper $P$-minuscule representations are restrictions of the actions from $\mathfrak{b}$ to $\mathfrak{b}'$ of \emph{Demazure modules} that correspond to $\lambda$-minuscule Weyl group elements for dominant integral weights $\lambda$.  Infinite dimensional upper $P$-minuscule representations are new and first appeared in \cite{Str,Unify}.  These representations share several properties with \textit{minuscule representations} of semisimple Lie algebras (i.e. irreducible highest weight representations whose weights are all in the Weyl group orbit of the highest weight).  It has been known since at least the 1980s that minuscule representations of semisimple Lie algebras can be realized using a similar poset construction; see also \cite{PaperClassifyMinuscule}.

We state one of our main results from \cite{Unify}. This result was obtained in the simply laced case as Theorem 38(a) of \cite{Unify} and in the general case as Theorem 6.1.1(a) of \cite{Str}.

\begin{theorem}\label{UnifyMain} 
Let $P$ be a $\Gamma$-colored poset.  Then $P$ is $\Gamma$-colored $d$-complete poset if and only if $\mathcal{FI}(P)$ carries an upper $P$-minuscule representation.
\end{theorem}

\noindent Applying Theorem \ref{TheoremClassify} obtains precisely when an upper $P$-minuscule representation can be built.



\begin{theorem}\label{TheoremUpperPMinusculeClassification}
Let $P$ be a $\Gamma$-colored poset.  Then $\mathcal{FI}(P)$ carries an upper $P$-minuscule representation if and only if $P$ is the disjoint union of connected posets in which each is either a finite slant sum of slant irreducible dominant minuscule heaps from \cite{DDCT,Ste} or an infinite filter of some connected full heap from \cite{Gre}.  
\end{theorem}


\noindent In general, there may be infinitely many possibilities for the diagonal operators $\{H_a\}_{a \in \Gamma}$ used to construct an upper $P$-minuscule representation from a $\Gamma$-colored $d$-complete poset.  The ``$\mu$-diagonal'' operators of \cite{Str,Unify} may always be chosen.  These operators are our preferred choice, since they are the unique operators providing the actions of $\{h_a\}_{a \in \Gamma}$ for the ``$P$-minuscule'' representations of $\mathfrak{g}'$.  They are also the unique choice of operators for upper $P$-minuscule representations of $\mathfrak{b}'$ when $P$ is finite.  The claims made in the previous three sentences are respectively confirmed in Theorems 30 and 35 and Corollary 29 of \cite{Unify} for the simply laced case and Theorems 5.3.1 and 5.4.2 and Corollary 4.5.2 of \cite{Str} for the general case.  Theorem \ref{TheoremUpperPMinusculeClassification} thus classifies the upper $P$-minuscule representations built using the $\mu$-diagonal operators.

\section*{Acknowledgements}
I would like to thank Robert A. Proctor for several helpful comments on notation, terminology, and exposition.  I would also like to thank Marc Besson and Sam Jeralds for helpful comments.

\bibliographystyle{amsplain}

\end{spacing}
\end{document}